\documentclass[11pt]{article}
\usepackage[english]{babel}
\usepackage{amsmath,amsthm}
\usepackage{amsfonts}
\usepackage[comma,numbers,square,sort&compress]{natbib}
\usepackage{graphicx,subfigure,amsmath,amssymb,mathrsfs,amsfonts,amstext,amsthm}
\usepackage{latexsym, amssymb}
\usepackage{graphicx}
\usepackage{subfigure}
\usepackage{pgf,tikz}
\usepackage{pstricks}
\newtheorem{thm}{Theorem}[section]
\newtheorem{cor}[thm]{Corollary}
\newtheorem{lem}[thm]{Lemma}

\theoremstyle{definition}
\newtheorem{defn}[thm]{Definition}
\theoremstyle{remark}
\newtheorem{rem}[thm]{Remark}
\numberwithin{equation}{section}
\begin{document}

\title{\bfseries\textrm{Partial Convergence of Heterogeneous Hegselmann-Krause Opinion Dynamics}}
%\author[su]{Wei Su}{\it suwei@amss.ac.cn}    % Add the
%\author[Yu]{Yongguang Yu}{\it ygyu@bjtu.edu.cn}               % e-mail address
%\address[su]{School of Science, Beijing Jiaotong University, Beijing 100044, China} % Please supply
%\address[Yu]{School of Science, Beijing Jiaotong University, Beijing 100044, China}
%\author{Wei Su\thanks{School of Science, Beijing Jiaotong University, Beijing 100044, China, {\tt
%suwei@amss.ac.cn}} \and Yongguang Yu\thanks{School of Science, Beijing Jiaotong University, Beijing 100044, China, {\tt
%ygyu@bjtu.edu.cn}}
%}
\author{Wei Su, ~~Yongguang Yu \thanks{{\tt suwei@amss.ac.cn, ygyu@bjtu.edu.cn}, School of Science, Beijing Jiaotong University, Beijing 100044, China}}
%\Author{Wei Su, Ge Chen And Yiguang Hong}%
%
%\address{}%
%\thanks{suwei@amss.ac.cn, ygyu@bjtu.edu.cn, School of Science, Beijing Jiaotong University, Beijing 100044, China}
\date{}%
\maketitle

% ----------------------------------------------------------------
% ----------------------------------------------------------------
% Article Class (This is a LaTeX2e document)  ********************
% ----------------------------------------------------------------
% ----------------------------------------------------------------
\begin{abstract}
In opinion dynamics, the convergence of the heterogeneous Hegselmann-Krause (HK) dynamics has always been an open problem for years which looks forward to any essential progress. In this short note, we prove a partial convergence conclusion of the general heterogeneous HK dynamis. That is, there must be some agents who will reach static states in finite time, while the other opinions have to evolve between them with a minimum distance if all the opinions does not reach consensus. And this result leads to the convergence of two special case of heterogeneous HK dynamics: the minimum confidence threshold is large enough, or the initial opinion difference is small enough.
\end{abstract}

\section{Introduction}

Opinion evolution dynamics are an important issue in the social dynamics and deeply involved with our social lives in a large aspect, including public voting and online forum.
In recent years, analysis on opinion dynamics has attracted more and more interests of many research areas such as mathematics, physics, psychology, and information technology \cite{Castellano2009}. Due to the complexity of social dynamics, various opinion phenomena, including agreement/consensus, crowd polarization, and opinion fluctuation, can be widely observed in reality. In order to quantitatively investigate the opinion evolution mechanisms, various multi-agent based models have been proposed, and among them one of the most famous models is the so called Hegselmann-Krause (HK) model \cite{krause,hegselmann2002}.

In the HK dynamics, the opinions of all agents are assumed to take values in a finite interval, and each agent has a confidence bound. All agents whose opinion values locate within the confidence bound of one agent are called its neighbors, and the agent updates its opinion value by averaging the opinion values of the current neighbors. According to the diversity of confidence bounds, the HK model can be divided into two cases: the homogeneous case and the heterogeneous case. The model is said to be homogeneous when all agents possess the same confidence bound, and heterogeneous otherwise. Though possesses simple-looking format, the HK model can effectively present the rich phenomena of real opinion evolution, and shows its own difficulty in theoretical analysis as well \cite{Lorenz2005,Lorenz2007}. One of the most basic concern of the HK dynamics is the convergence of the opinions. For the homogeneous case, it has been proved that the opinions will converge to static state in finite time \cite{Lorenz2005,Blondel2009}, while the convergence of the heterogeneous case has always been an open problem for years \cite{Mirtabatabaei2012}, although it is supported by vast simulation results. Up to now, only a few results have appeared. In \cite{Etesami2015}, a sufficient condition for the heterogeneous HK systems to converge is given, and in \cite{Chazelle2015}, the inertial HK systems (besides the homogeneous agents, the system has some close-mind agents whose confidence bounds are 0) is proved to converge. Other than these pleasurable but limited progress, there is still no direct conclusion on the convergence of the general heterogeneous HK dynamics, and any evident theoretical progress is anticipated.

In this short note, we prove a partial convergence result about the general heterogenous HK dynamics. Here the heterogeneous confidence bounds are supposed to be positive, and it is proved that there must be some opinions that reach steady states in finite time, while the other opinions have to stay between them with a distance of the minimum confidence bound if all the opinions does not reach consensus. Using this result, we obtain the convergence of two special heterogeneous HK models. That is, when the initial opinion difference is no larger than the twice of the minimum confidence bound, or all the confidence bounds are not less than half the opinion interval, the heterogeneous HK dynamics will reach convergence.

The rest of this paper is organized as follows: Section II gives the scheme of the heterogeneous HK dynamics and the problem formulation, Section III presents the main results, and Section IV gives some concluding remarks.

\section{Preliminaries and Formulation}
\renewcommand{\thesection}{\arabic{section}}
%%%%%%%%%%%%%%%%%%%%%%%%%%%%%%%%%%%%%%%%%%%%%%%%%%%%%%%%%%%%%%%%%%%%%%%%%%%%%%%%%%%%%%%%%%%%%%%%%%%%
Suppose there are $n$ agents, whose opinion values at time $t$ are denoted by $x(t)$, which take values in $[0,1]^n$,
then the heterogeneous HK model is described as follows:
\begin{equation}\label{heterHKmodel}
  x_i(t+1)=\frac{1}{|\mathcal{N}_i(x(t))|}\sum\limits_{j\in\mathcal{N}_i(x(t))}x_j(t),\,\,\,i\in\mathcal{V}=\{1,\ldots,n\},
\end{equation}
where $x_i(t)\in [0, 1], \, i\in \mathcal{V}$ is the opinion value of agent $i$ at time $t$ and
\begin{equation}\label{neigh}
 \mathcal{N}_i( x(t))=\{1\leq j\leq n\; \big|\; |x_j(t)-x_i(t)|\leq r_i\}
\end{equation}
with $r_i\in(0,1]$ representing the confidence bound of agent $i$.
\begin{rem}
On consideration of the relevance to the convergence results of the two special heterogeneous HK dynamics obtained in this paper, the confidence bounds in system (\ref{heterHKmodel}) are supposed to be positive. And when $r_i=r, i\in\mathcal{V}$ for some $r\in(0,1]$, system (\ref{heterHKmodel}) degenerates to the homogeneous case.
\end{rem}
For the homogeneous HK dynamics, it is proved
\begin{lem}\label{HKfrag}\cite{Blondel2009}
The opinion values of homogeneous dynamics (\ref{heterHKmodel})~are order-preserving, i.e., if $x_i(0)\leq x_j(0),\,\forall i,j$, then
$ x_i(t)\leq x_j(t),\,\,\,t\geq 0$. Moreover, for every $1\leq i\leq n$, $x_i(t)$ will converge to some $x_i^*$ in finite time, and either $x_i^*=x_j^*$ or $|x_i^*-x_j^*|> r$ holds for any $i, j$.
\end{lem}
Here, the convergence of an opinion can be defined as follows:
\begin{defn}
The opinion value $x_i(t)$ of agent $i$ is said to converge if there exists $x_i^*\in[0,1]$ such that $\lim_{t\rightarrow \infty}x_i(t)=x_i^*$; and converge in finite time, if there exists time $T\geq 0$ such that $x_i(t)=x_i^*$ for $t\geq T$.
\end{defn}
Lemma \ref{HKfrag} says that the opinion of the homogeneous HK model will converge in finite time, but it is more complex for the heterogeneous case. Though vast amount of simulations support the convergence of the heterogeneous HK model, systems can be easily given that the convergence may take place in an asymptotical way, rather than in finite time. For example, suppose $\mathcal{V}=\{1,2,3\}, x_i(0)=0.1, 0.4, 0.8, r_i=0.1, 0.4, 0.1$ respectively, then $x_1(t)=0.1, x_3(t)=0.8$ for $t\geq 0$, but $x_2(t)$ tends to $0.45$ asymptotically.

The order-preserving property and the monotone change of the boundary opinions of the homogeneous HK model can be obtained by the following lemma:
\begin{lem}\label{monosmlem}\cite{Blondel2009}
Suppose $\{z_i, \, i=1, 2, \ldots\}$ is a nonnegative nondecreasing (nonincreasing) sequence.  Then for any integer $m>0$, the sequence
\begin{equation*}
 g_m(k)=\frac{1}{k}\sum\limits_{i=m+1}^{m+k}z_i, \, k=1, 2, \ldots,
\end{equation*}
is monotonously nondecreasing (nonincreasing) with respect to $k$.
\end{lem}

\section{Main Results}\label{Consensus_condition}
\renewcommand{\thesection}{\arabic{section}}
\subsection{Partial Convergence of General Heterogeneous HK Systems}

\begin{thm}\label{partconver}
Let $r=\min_{i\in \mathcal{V}}r_i$, then for system (\ref{heterHKmodel}), there exist $c,c'\in[0,1]$ with $c\leq c'$, $\mathcal{C},\mathcal{C}'\subset \mathcal{V}$ with $|\mathcal{C}|\geq 1,|\mathcal{C}'|\geq 1$ and $t^*\geq 0$, such that\\
~~~(a) $x_i(t)=c,x_j(t)=c'$ for $i\in\mathcal{C},j\in\mathcal{C}',t\geq t^*$;\\
~~~(b) the convergence is reached in finite time, if $c'-c\leq 2r$;\\
~~~(c) $x_k(t)-c>r, c'-x_k(t)>r$ for $k\in \mathcal{V}-\mathcal{C}\bigcup\mathcal{C}',t\geq t^*$, if $c'-c> 2r$.
\end{thm}
\begin{proof}
Define functions
\begin{equation*}
  x_{max}(t)=\max\limits_{i\in \mathcal{V}}x_i(t),\,\,\,x_{min}(t)=\min\limits_{i\in \mathcal{V}}x_i(t)
\end{equation*}
for $t\geq 0$.
It is easy to see that $x_{max}(t)$ is nonincreasing and $x_{min}(t)$ is nondecreasing.

Since \{$x_{min}(t), t\geq 0\}$ is monotonically nondecreasing, and has an upper bound 1, thus a real number $c\in [0,1]$ exists such that
\begin{equation}\label{limitmin}
  \lim\limits_{t\rightarrow \infty}x_{min}(t)=c.
\end{equation}
This implies for $\forall\epsilon>0$, there exists $T\geq 0$, such that $c-\epsilon<x_{min}(t)\leq c$ for $t>T$.

Take $\epsilon\in(0,\frac{r}{n^2})$, $T\geq 0$ such that $c-\epsilon<x_{min}(T)\leq c$. Denote $\mathcal{V}_\epsilon^1(t)=\{i\in\mathcal{V}|c-\epsilon<x_i(t)\leq c+(n-1)\epsilon\}, t\geq 0$. If $|\mathcal{V}_\epsilon^1(T)|=n$, then $x_{max}(T)-x_{min}(T)<n\epsilon\leq \frac{r}{n}$, which means all agents at time $T$ are neighbors to each other, and hence the opinions converge to a common average value at time $T+1$. Otherwise, $|\mathcal{V}_\epsilon^1(T)|<n$. If there exists $j\notin \mathcal{V}_\epsilon^1(T)$ and $c+(n-1)\epsilon<x_j(T)\leq c+r-\epsilon$, then for $i\in \mathcal{V}_\epsilon^1(T)$, we have $j\in\mathcal{N}_i( x(T))$ and by Lemma \ref{monosmlem},
\begin{equation*}
\begin{split}
  x_i(T+1) &  =\frac{1}{|\mathcal{N}_i( x(T))|}\sum\limits_{k\in\mathcal{N}_i( x(T))}x_k(T)\\
    & \geq \frac{1}{|\mathcal{V}_\epsilon^1(T)|+1}\bigg(\sum\limits_{k\in\mathcal{V}_\epsilon^1(T)}x_k(T)+x_j(T)\bigg)\\
    & > \frac{1}{|\mathcal{V}_\epsilon^1(T)|+1}(|\mathcal{V}_\epsilon^1(T)|(c-\epsilon)+c+(n-1)\epsilon)\\
    & = c-\epsilon+ \frac{n}{|\mathcal{V}_\epsilon^1(T)|+1}\epsilon\geq c,
\end{split}
\end{equation*}
which contradicts with (\ref{limitmin}). Thus for $j\notin \mathcal{V}_\epsilon^1(T)$, we have $x_j(T)>c+r-\epsilon$ and by repeating the above discussion, we have
\begin{equation}\label{distant}
  x_j(t)>c+r-\epsilon,\,\,\mbox{if}\,\,j\notin\mathcal{V}_\epsilon^1(t),\,\, \mbox{for}\,\,t\geq T.
\end{equation}
For $j\notin\mathcal{V}_\epsilon^1(T)$, by Lemma \ref{monosmlem} and (\ref{distant}), we have
\begin{equation*}
\begin{split}
  x_j(T+1)&  =\frac{1}{|\mathcal{N}_j( x(T))|}\sum\limits_{k\in\mathcal{N}_j( x(T))}x_k(T)\\
  & \geq \frac{1}{|\mathcal{N}_j( x(T))|}\bigg(\sum\limits_{k\in\mathcal{V}_\epsilon^1(T)}x_k(T)+\sum\limits_{k\notin\mathcal{V}_\epsilon^1(T)}x_k(T)\bigg)\\
  & > \frac{1}{|\mathcal{N}_j( x(T))|}\bigg((|\mathcal{N}_j(x(T))|-1)(c-\epsilon)+c+r-\epsilon\bigg)\\
  & = c+\frac{r}{|\mathcal{N}_j(x(T))|}-\epsilon\geq c+\frac{r}{n}-\epsilon>c+(n-1)\epsilon,
\end{split}
\end{equation*}
which means $j\notin\mathcal{V}_\epsilon^1(T+1)$. Consider (\ref{distant}) again and repeat the above discussion, we have
\begin{equation}\label{jnotin}
  j\notin\mathcal{V}_\epsilon^1(t),\,\,\, t\geq T.
\end{equation}
This means that if one agent $j$ is not in the set $\mathcal{V}_\epsilon^1(T)$, it will never get in it.
For $i\in \mathcal{V}_\epsilon^1(T)$, if there exist $j\in\mathcal{N}_i( x(T))-\mathcal{V}_\epsilon^1(T)$, then by (\ref{distant}) and Lemma \ref{monosmlem},  we have
\begin{equation}\label{igetout}
  \begin{split}
  x_i(T+1) &  =\frac{1}{|\mathcal{N}_i( x(T))|}\sum\limits_{k\in\mathcal{N}_i( x(T))}x_k(T)\\
    & \geq \frac{1}{|\mathcal{V}_\epsilon^1(T)|+1}\bigg(\sum\limits_{k\in\mathcal{V}_\epsilon^1(T)}x_k(T)+x_j(T)\bigg)\\
    & > \frac{1}{|\mathcal{V}_\epsilon^1(T)|+1}(|\mathcal{V}_\epsilon^1(T)|(c-\epsilon)+c+r-\epsilon)\\
    & = c+\frac{r}{|\mathcal{V}_\epsilon^1(T)|+1}-\epsilon\geq c+\frac{r}{n}-\epsilon>c+(n-1)\epsilon.
\end{split}
\end{equation}
This means if one agent $i$ in $\mathcal{V}_\epsilon^1(T)$ has a neighbor which is not in $\mathcal{V}_\epsilon^1(T)$, then $i$ will leave $\mathcal{V}_\epsilon^1(T)$ at the next time, and by (\ref{jnotin}), we know that $\mathcal{V}_\epsilon^1(t)\supset\mathcal{V}_\epsilon^1(t+1)$ for $t\geq T$. Since $|\mathcal{V}_\epsilon^1(T)|<n$, and $|\mathcal{V}_\epsilon^1(t)|\geq 1, t\geq T$ by (\ref{limitmin}), there exists time $T^*\geq T$ such that for $t\geq T^*$,
\begin{equation}\label{noneighout}
  \mathcal{N}_i(x(t))\bigcap(\mathcal{V}-\mathcal{V}_\epsilon^1(t))=\varnothing,\,\,\,i\in \mathcal{V}_\epsilon^1,
\end{equation}
which means all the agents in $\mathcal{V}_\epsilon^1(t)$ have no neighbor which is not in $\mathcal{V}_\epsilon^1(t)$. By (\ref{limitmin}), we have
\begin{equation}\label{limitset}
  x_i(t+1)=c,\,\,\,i\in\mathcal{V}_\epsilon^1(T^*),\,\,\,t\geq T^*.
\end{equation}
Thus by (\ref{jnotin}), (\ref{igetout}) and (\ref{limitset}), we know that for every $i\in \mathcal{V}_\epsilon^1(T^*)$, it never has neighbors outside $\mathcal{V}_\epsilon^1(T^*)$. Similarly, there exists $c^{'}\in[c,1]$ and $T^{*'}\geq 0$ such that
$\lim\limits_{t\rightarrow \infty}x_{max}(t)=c^{'}$, and
\begin{equation}\label{limitset2}
  x_i(t+1)=c^{'},\,\,\,i\in\mathcal{V}_\epsilon^2(T^{*'}),\,\,\,t\geq T^{*'}
\end{equation}
where $\mathcal{V}^2_\epsilon(t)=\{i|c^{'}-(n-1)\epsilon\leq x_i(t)<c^{'}+\epsilon\}, t\geq 0$, and all the agents in $\mathcal{V}_\epsilon^2(T^{*'})$ have no neighbor which is not in $\mathcal{V}_\epsilon^2(T^{*'})$. Let $\mathcal{C}=\mathcal{V}^1_\epsilon(T^*), \mathcal{C}'=\mathcal{V}^2_\epsilon(T^{*'})$, then conclusion (a) holds.
Denote $\mathcal{V}_\epsilon=\mathcal{C}\bigcup\mathcal{C}'$. Take $t\geq \max\{T^*,T^{*'}\}$, if $|c'-c|\leq 2r$, then $\mathcal{V}_\epsilon=\mathcal{V}$, and the convergence holds. Or there is $k\in\mathcal{V}-\mathcal{V}_\epsilon^1(t)$, and either $x_k(t)-c\leq r$ or $c'-x_k(t)\leq r$ holds, which contradicts with (\ref{noneighout}). If $c'-c>2r$,
(\ref{limitset}) and (\ref{limitset2}) indicate that the agents in $\mathcal{V}_\epsilon$ reach static states, while all agents that are not in $\mathcal{V}_\epsilon$ evolve in the interval $(c+\bar{r},c^{'}-\bar{r})$ with $\bar{r}=\min\limits_{i\in \mathcal{V}_\epsilon} r_i$ since by (\ref{noneighout}) they cannot be neighbors of the agents in $\mathcal{V}_\epsilon$, thus conclusion (c) holds.
\end{proof}
\begin{cor}\label{spec1}
For the heterogeneous HK system (\ref{heterHKmodel}), if the initial opinions satisfy $x_{\max}(0)-x_{\min}(0)\leq 2r$, then (\ref{heterHKmodel}) will converge in finite time.
\end{cor}
\begin{proof}
Since $x_{\max}(0)-x_{\min}(0)\leq 2r$, by Lemma \ref{monosmlem} and Theorem \ref{partconver}, $c'-c\leq 2r$, then the conclusion holds by Theorem \ref{partconver} (b).
\end{proof}
\begin{cor}\label{spec2}
For the heterogeneous HK system (\ref{heterHKmodel}), if $r\geq 0.5$, then for any initial opinion values, (\ref{heterHKmodel}) will converge in finite time.
\end{cor}
\begin{proof}
If $r\geq 0.5$, we can get that $x_{\max}(0)-x_{\min}(0)\leq1\leq 2r$, then the conclusion holds by Corollary \ref{spec1}.
\end{proof}

\section{Conclusions}\label{Conclusions}

Convergence of the heterogeneous Hegselmann-Krause opinion dynamics has always been an open problem for years. In this paper, a partial convergence result is proved, where it is shown that some opinions must reach convergence in finite time, and the other opinions will stay between them with a distance of the minimum confidence bound if the whole system does not reach convergence. After that, the convergence of two special heterogeneous HK model is obtained, that is, when the initial opinion difference is small enough, or the confidence bound is large enough, the heterogeneous HK dynamics will reach convergence.

\end{document}